\documentclass{article}
\usepackage{graphicx} 
\usepackage{amsmath, amssymb, amsthm}
\usepackage[hidelinks]{hyperref}
\usepackage{url}
\usepackage{authblk}

\newtheorem{theorem}{Theorem}

\newtheorem{lemma}[theorem]{Lemma}
\newtheorem{claim}[theorem]{Claim}

\newtheorem{corollary}[theorem]{Corollary}

\newcommand{\N}{\mathbb{N}}

\title{A note on infinite versions of $(p,q)$-theorems\thanks{Both authors were supported by the ERC Advanced Grant ``ERMiD'' and by the Thematic Excellence Program TKP2021-NKTA-62 of the National Research, Development and Innovation Office. The first author was also supported by the NKFIH grants FK132060 and SNN135643. The second author was also supported by the J\'anos Bolyai Research Scholarship of the Hungarian Academy of Sciences and by \'UNKP-23-5 of NRDIO.}}
\author[1,2]{Attila Jung}
\author[1,2]{Dömötör Pálvölgyi} 
\affil[1]{ELTE Eötvös Loránd University, Budapest, Hungary}
\affil[2]{HUN-REN Alfréd Rényi Institute of Mathematics, Budapest, Hungary}

\begin{document}

\maketitle

\begin{abstract}
    We prove that fractional Helly and $(p,q)$-theorems imply $(\aleph_0,q)$-theorems in an entirely abstract setting.
    We give a plethora of applications, including reproving almost all earlier $(\aleph_0,q)$-theorems about geometric hypergraphs that were proved recently. 
    Some of the corollaries are new results, for example, we prove that if $\mathcal{F}$ is an infinite family of convex compact sets in $\mathbb{R}^d$ and among every $\aleph_0$ of the sets some $d+1$ contain a point in their intersection with integer coordinates, then all the members of $\mathcal{F}$ can be hit with finitely many points with integer coordinates.
\end{abstract}

\section{Introduction}\label{sec:intro}

Let $\mathcal{K}_d$ be the hypergraph whose vertices are the compact convex sets in $\mathbb{R}^d$, and edges represent intersecting families of convex sets; note that the edges form a downwards closed set system. Many results of combinatorial convexity can be stated as properties of this hypergraph, called nerve complex in topology.

For a hypergraph $\mathcal{H}$, denote its vertex set by $V(\mathcal{H})$, the number of its edges by $e(\mathcal{H})$, and the $q$-uniform part, consisting of the edges that contain exactly $q$ vertices, by $\mathcal{H}^{(q)}$.
For an $S \subset V(\mathcal{H})$ vertex set, let $\mathcal{H}[S]$ be the subhypergraph spanned by $S$, which consists of the edges contained entirely in $S$.  

According to the celebrated Alon-Kleitman $(p,q)$-theorem, for every $p\ge d+1$, for every family $\mathcal{F}$ of compact convex sets in $\mathbb{R}^d$, if among every $p$ members of $\mathcal{F}$ some $d+1$ are intersecting (i.e., they all have a point in common), then all the members of $\mathcal{F}$ can be hit by $C=C(p,d)$ points.
In our language, this can be stated as follows.

\begin{theorem}[Alon and Kleitman \cite{alon1992piercing}]\label{thm:pq}
For every finite $p \geq d+1$ there exists a $C<\infty$ with the property that if $S \subset V(\mathcal{K}_d)$ is such that $\mathcal{K}_d^{(d+1)}[S]$ does not contain independent sets of size $p$, then $S$ can be covered with $C$ edges of $\mathcal{K}_d$.
\end{theorem}

One of the main ingredients in its proof, the Katchalski-Liu fractional Helly theorem, can be phrased as follows.

\begin{theorem}[Katchalski and Liu \cite{katchalski1979problem}]\label{thm:fh}
    If $S \subset V(\mathcal{K}_d)$ is a finite subset and $e(\mathcal{K}_d^{(d+1)}[S]) \geq \alpha \binom{|S|}{d+1}$ for some $\alpha > 0$, then there exists an edge of $\mathcal{K}_d[S]$ of size $\beta |S|$ where $\beta=\beta(\alpha,d) > 0$ depends only on $\alpha$ and $d$.
\end{theorem}

More generally, we say that a $q$-uniform (infinite)\footnote{Note that the fractional Helly property is true for all finite hypergraph, as we can choose a small enough $\beta$. Instead, we could define the fractional Helly property for a family of finite hypergraphs, but it is easier to deal only with one infinite hypergraph.} hypergraph $\mathcal{H}$ satisfies the fractional Helly property, if for all $\alpha>0$ there exists a $\beta>0$ such that if $e(\mathcal{H}[S]) \geq \alpha \binom{|S|}{q}$ for some finite $S \subset V(\mathcal{H})$, then there exists a ($q$-uniform) clique of $\mathcal{H}[S]$ of size  $\beta |S|$. Because of Helly's theorem, Theorem~\ref{thm:fh} is equivalent to that $\mathcal{K}_d^{(d+1)}$ satisfies the fractional Helly property.\\

For $0 \leq k < d$, let $\mathcal{B}_{d,k}$ be the hypergraph whose vertices are compact balls from $\mathbb{R}^d$, and edges represent families of balls which can be intersected with a single $k$-flat ($k$-dimensional affine subspace). Keller and Perles proved an infinite variant of the Alon-Kleitman theorem for $k$-flats intersecting Euclidean balls, that states that if we are given a collection of closed balls $S$ such that among any infinite subcollection of $S$ there are $k+2$ that can be intersected with a single $k$-flat, then there are finitely many $k$-flats that stab all balls of $S$. In our language, this can be stated as follows.

\begin{theorem}[Keller and Perles \cite{keller2022aleph_0}]\label{thm:kellerperles}
    If $S \subset V(\mathcal{B}_{d,k})$ is such that $\mathcal{B}_{d,k}^{(k+2)}[S]$ has no infinitely large independent set, then $S$ can be covered with a finite number of edges of $\mathcal{B}_{d,k}$.
\end{theorem}

Theorem \ref{thm:kellerperles} was proved for $k=0$ for general balls, and for $k > 0$ for only unit balls in the first version of \cite{keller2022aleph_0} that appeared in SoCG 2022, and later for all $0 \leq k < d$ for general radius balls that can be found in their arXiv preprint. We prove that such an infinite variant of the Alon-Kleitman theorem always follows from the corresponding finite version and a fractional Helly theorem. In fact, we prove that if our hypergraph satisfies the fractional Helly property, then the condition of the infinite variant of the Alon-Kleitman theorem implies the condition of the finite version with some finite $p$. We state this in the contrapositive form as follows.

\begin{theorem}\label{thm:main}
    If a $q$-uniform hypergraph satisfies the fractional Helly property and has arbitrarily large finite independent sets, then it has an infinitely large independent set.
\end{theorem}

Theorem~\ref{thm:main} is proved in Section~\ref{sec:proof}. In the rest of the introduction, we list a few corollaries of it, all analogs of the result of Keller and Perles, and in Section~\ref{subsec:colorful} we discuss colorful extensions.\\

Combining Theorem~\ref{thm:main} with Theorems~\ref{thm:pq} and \ref{thm:fh}, we get the following result.

\begin{corollary}\label{cor:alephMeetsClassical}
    Let $\mathcal{F}$ be a family of compact convex sets in $\mathbb{R}^d$. If among every $\aleph_0$ members of $\mathcal{F}$ some $d+1$ are intersecting, then all the members of $\mathcal{F}$ can be hit by finitely many points.
\end{corollary}

In our language, this can be stated as follows.
    If $S \subset V(\mathcal{K}_d)$ is such that $\mathcal{K}_d^{(d+1)}[S]$ has no infinitely large independent set, then $S$ can be covered with a finite number of edges of $\mathcal{K}_d$.

\begin{proof}[Proof of Corollary~\ref{cor:alephMeetsClassical}]
    By Theorem~\ref{thm:fh}, $\mathcal{K}_d$ satisfies the fractional Helly property. As $\mathcal{K}_d^{(d+1)}[S]$ has no infinitely large independent set, there exists some finite $p$ such that $\mathcal{K}_d^{(d+1)}[S]$ has no independent set of size $p$ by Theorem~\ref{thm:main}. We can use Theorem~\ref{thm:pq} to conclude that $S$ can be covered with $C(p,d)$ edges of $\mathcal{K}_d$.
\end{proof}

Similarly, the fractional Helly and $(p,q)$-theorems about hyperplanes intersecting convex sets by Alon and Kalai \cite{alon1995bounding} imply the following infinite variant.

\begin{corollary}\label{cor:AlonKalai}
    Let $\mathcal{F}$ be a family of compact convex sets in $\mathbb{R}^d$. If among every $\aleph_0$ members of $\mathcal{F}$ some $d+1$ can be hit by a hyperplane, then all the members of $\mathcal{F}$ can be hit by finitely many hyperplanes.
\end{corollary}

The fractional Helly and $(p,d+1)$-theorems about convex lattice sets by Bárány and Matou{\v s}ek imply the following infinite variant.

\begin{corollary}\label{cor:lattice}
    Let $\mathcal{F}$ be a family of compact convex sets in $\mathbb{R}^d$. If among every $\aleph_0$ member of $\mathcal{F}$ some $d+1$ contain a point in their intersection with integer coordinates, then all the members of $\mathcal{F}$ can be hit by finitely many points with integer coordinates.
\end{corollary}

A family of sets in $\mathbb{R}^d$ is said to have bounded description complexity, if each of the sets is definable with a bounded number of polynomial inequalities of bounded degree. Matou{\v s}ek proved fractional Helly and $(p, (d-k)(k+1)+1)$-theorems for $k$-flats intersecting $d$-dimensional sets of bounded description complexity. Theorem~\ref{thm:main} implies the following corollary.

\begin{corollary}\label{cor:boundedComplexity}
    Let $\mathcal{F}$ be a family of sets in $\mathbb{R}^d$ of bounded description complexity. If among every $\aleph_0$ member of $\mathcal{F}$ some $(d-k)(k+1)+1$ can be hit with a single $k$-dimensional affine subspace, then all the members of $\mathcal{F}$ can be hit by finitely many $k$-dimensional affine subspaces.
\end{corollary}

Chakraborty, Ghosh and Nandi \cite{chakraborty2023stabbing} recently proved an $(\aleph_0, 2)$-theorem for axis-parallel boxes and axis-parallel $k$-flats, which we state later as Theorem \ref{thm:boxesFlats}. The below $k=d-1$ case can be easily proved with our method, and was already stated earlier in the previous version of our manuscript.

\begin{corollary}\label{cor:chakraborty}
    Let $\mathcal{F}$ be a family of axis-parallel boxes in $\mathbb{R}^d$. If among every $\aleph_0$ of them some $2$ can be intersected with an axis-parallel hyperplane, then all the members of $\mathcal{F}$ can be hit by finitely many hyperplanes.
\end{corollary}

For the proof, we only need suitable versions of the fractional Helly property and a $(p,2)$-theorem for boxes and hyperplanes to be able to use Theorem~\ref{thm:main}.
These suitable versions follow in a fairly standard way from known methods, which we omit here, because later we give a full proof of the more general Theorem \ref{thm:boxesFlats}.
For $k<d-1$, a slightly stronger form of Theorem~\ref{thm:main} (Theorem~\ref{thm:mainGeneralized}) and some claims about finite families of boxes will be needed, which we discuss in Section~\ref{sec:boxes}.\\

Using the fractional Helly and $(p,k+2)$-Theorems about $k$-flats intersecting Euclidean balls proved in \cite{jung2024fat}, Theorem~\ref{thm:main} also provides an alternative proof of Theorem~\ref{thm:kellerperles} of Keller and Perles~\cite{keller2022aleph_0}, which initiated this whole line of research.\\

We could give a long list of other corollaries of Theorem~\ref{thm:main}, one for each case where a fractional Helly and a $(p,q)$-type result is known, but such a list would add little to the paper.
We only show some colorful variants of some of the results discussed above in Subsection~\ref{subsec:colorful}. Section~\ref{sec:proof} contains the proof of our main Theorem~\ref{thm:main}. At the end of the proof, in Subsection~\ref{subsec:final} we state Theorem~\ref{thm:mainGeneralized}, a strengthening of Theorem~\ref{thm:main}, where the assumption about the fractional Helly property is weakened. We show an application of this slightly stronger statement in Section~\ref{sec:boxes}.

\subsection{Colorful variants}\label{subsec:colorful}

Bárány, Fodor, Montejano, Oliveros, and Pór \cite{barany2014colourful} proved a colorful variant of Theorem~\ref{thm:pq} of Alon and Kleitman \cite{alon1992piercing}. For a given sequence $\mathcal{F}_1, \mathcal{F}_2, \ldots$ of families, a sequence of elements $C_1, C_2, \ldots$ is \textit{heterochromatic} if there exists $i_1 < i_2 < \ldots$ with $C_j \in \mathcal F_{i_j}$.

\begin{theorem}[Bárány, Fodor, Montejano, Oliveros, Pór \cite{barany2014colourful}]\label{thm:BaranyFodor}
    For every positive integers $p \geq d+1$ there exists a positive integer $C(p,d)$ such that the following holds. Let $\mathcal{F}_1,\mathcal{F}_2, \ldots, \mathcal{F}_p$ be families of compact convex sets in $\mathbb{R}^d$. If every heterochromatic sequence of convex sets of length $p$ contains $d+1$ intersecting sets, then there exists an $\mathcal{F}_i$ which can be pierced by $C(p,d)$ points.
\end{theorem}

This implies the following infinite variant.

\begin{corollary}\label{cor:BaranyFodor}
    Let $\mathcal{F}_1,\mathcal{F}_2, \ldots$ be families of compact convex sets in $\mathbb{R}^d$. If every infinite heterochromatic sequence of convex sets contains $d+1$ intersecting sets, then there exists an $\mathcal{F}_i$ which can be pierced by finitely many points.
\end{corollary}

\begin{proof}[Proof of Corollary~\ref{cor:BaranyFodor}]
    Suppose that no $\mathcal{F}_i$ can be pierced by finitely many points. By Theorem~\ref{thm:BaranyFodor}, for every finite $p$ and every $\mathcal{F}_{i_1}, \ldots, \mathcal{F}_{i_p}$ there exists a heterochromatic sequence $C_1 \in \mathcal{F}_{i_1}, \ldots, C_p \in \mathcal{F}_{i_p}$ without $d+1$ intersecting set. This implies that there exists an infinite heterochromatic sequence which contains arbitrarily large finite subsequences without $d+1$ intersecting sets. By Theorems~\ref{thm:fh} and \ref{thm:main} we can find an infinite heterochromatic subsequence without $d+1$ intersecting sets, which contradicts the assumption of Corollary~\ref{cor:BaranyFodor}.
\end{proof}

By replacing Theorems~\ref{thm:fh} and \ref{thm:BaranyFodor} by other fractional Helly and colorful $(p,q)$-results, our Theorem~\ref{thm:main} can be applied to other colorful variants as well. We mention only one of them, which is about $k$-flats intersecting Euclidean balls. Its weaker analog (for balls whose radius is in the range $[r,R]$) was proved by Ghosh and Nandi~\cite{ghosh2022heterochromatic}.

\begin{corollary}\label{cor:CKellerPerles}
    Let $\mathcal{F}_1,\mathcal{F}_2, \ldots$ be families of closed balls in $\mathbb{R}^d$. If every heterochromatic sequence of balls contains $k+2$ balls with a $k$-transversal (i.e., there is a $k$-flat stabbing all $k+2$ balls), then there exists an $\mathcal{F}_i$ which can be pierced by finitely many $k$-flats.
\end{corollary}

\begin{proof}
    Combine Theorems 5 and 8 from \cite{jung2024fat} with Theorem~\ref{thm:main} as in the proof of Corollary~\ref{cor:BaranyFodor}.
\end{proof}

Keller and Perles (personal communication) proved the same corollary independently with a different method.
After the first version of our manuscript appeared online, Chakraborty, Ghosh and Nandi~\cite{chakraborty2024heterochromatic} also derived the above corollary, with methods similar to Keller and Perles \cite{keller2022aleph_0}.
Their proof also works if we replace balls with convex sets that are not too elongated; we cannot prove such results with our methods because of a lack of weak epsilon nets for the hypergraph given by such sets stabbed by $k$-flats.\\

The following example, which we also learned from Keller, shows that we cannot weaken the assumption of Corollary~\ref{cor:CKellerPerles} and only require that every heterochromatic sequence of balls with exactly one ball from all the families contains a subsequence of $k+2$ balls with a $k$-transversal (as we incorrectly claimed in the first version of our manuscript). Let $\mathcal{F}_0$ be an infinite family $\{B_1, B_2, \ldots\}$ of balls such that no $k+2$ of them have a $k$-transversal and for $i > 0$ let $\mathcal{F}_i$ be a family of balls inside $B_i$ such that no $k+2$ of them have a $k$-transversal. In this case, there is no $\mathcal{F}_i$ which can be pierced by finitely many $k$-flats, but no matter how we choose one ball from each of the families, the ball chosen from $\mathcal{F}_0$ will intersect another chosen ball, thus some $k+2$ of the chosen balls will have a $k$-transversal.
A one-dimensional version of the example shows that the assumption of Corollary~\ref{cor:BaranyFodor} also cannot be weakened to only assume $d+1$ intersecting sets in heterochromatic sequences with exactly one set from each family.

\section{Proof of Theorem~\ref{thm:main}}
\label{sec:proof}

In Section~\ref{subsec:Mqt}, we show a class $M_s^{(q)}(t)$ of forbidden subhypergraphs of hypergraphs satisfying the fractional Helly property. In Section~\ref{subsec:superhom}, we prove a lemma about finding highly homogeneous subhypergraphs in infinite hypergraphs. Finally, in Section~\ref{subsec:final}, we prove Theorem~\ref{thm:main} by showing how homogeneous $M_s^{(q)}(t)$-free hypergraphs with arbitrary large finite independent sets contain infinitely large independent sets.

\subsection{A consequence of the fractional Helly property}\label{subsec:Mqt}

Motivated by Holmsen \cite{holmsen20}, for any $s,t\ge q$ we define $M_s^{(q)}(t)$ as a class of $q$-uniform hypergraphs as follows. Take $st$ vertices divided into $s$ parts of size $t$ such that we have a complete $q$-uniform $s$-partite hypergraph among the parts, but no edge inside any part. There is no restriction on the ``mixed'' edges that intersect more than one, but less than $q$ parts. 
If $q=2$, there are no mixed edges, the only  graph in the family $M_s^{(2)}(t)$ is the complete $s$-partite graph $K_{t,\ldots, t}$. 
For $q>2$, however, there are several different $q$-uniform graphs in $M_s^{(q)}(t)$. We call a $q$-uniform hypergraph $M_s^{(q)}(t)$-free if it contains none of them as an induced subgraph. 
By monotonicity, if a hypergraph is $M_s^{(q)}(t)$-free, it is also $M_s^{(q)}(t+1)$-free and $M_{s+1}^{(q)}(t)$-free.

Holmsen \cite{holmsen20} proved that for any $s\ge q$, the $M_s^{(q)}(q)$-free $q$-uniform hypergraphs satisfy the fractional Helly property,
which can be interpreted as that the `fractional Helly number' of any hypergraph is 
at most as large as the `colorful Helly number' of the hypergraph (see also \cite{holmsen21}). In the opposite direction, we observe the following.

\begin{claim}\label{cl:Mqt}
    Every $q$-uniform hypergraph that has the fractional Helly property is $M_q^{(q)}(t)$-free for $t> \frac{q-1}{\beta}$ where $\beta$ belongs to $\alpha=\frac{q!}{q^q}$. 
\end{claim}
\begin{proof}
    A graph from $M_q^{(q)}(t)$ would have $t^q\ge \frac{q!}{q^q}\binom{qt}{q}$ edges, so by the fractional Helly property would contain a clique of size $\beta qt$, but the largest clique in any graph from $M_q^{(q)}(t)$ has size at most $(q-1)q$.
\end{proof}

Thus, $M_s^{(q)}(q)$-freeness implies the fractional Helly property for any $s \geq q$ by \cite{holmsen20}, and the fractional Helly property implies $M_q^{(q)}(t)$-freeness for some $t \geq q$ by Claim~\ref{cl:Mqt}. Although it will not be used in the proof of Theorem~\ref{thm:main}, for the sake of completeness, we observe that neither of the above implications can be reversed.

\begin{claim}
    There are hypergraphs with the fractional Helly property which are not $M_s^{(q)}(q)$-free for any $s \geq q$, and there are $M_q^{(q)}(q+1)$-free hypergraphs without the fractional Helly property.
\end{claim}

\begin{proof}
    For the first part, let $H$ be 
    the complement of an infinite matching, i.e., the complement of infinitely many pairwise disjoint edges containing $q$ vertices each. This is clearly not $M_s^{(q)}(q)$-free for any $s \geq q$, and any $n$ vertices contain a clique of size at least $\frac{q-1}{q}n$, so $\beta=\frac{q-1}{q}$ is a good choice for every $\alpha$.

    For the second part, for any $n$ we construct a hypergraph $H_n$ which is $M_q^{(q)}(q+1)$-free, dense, but has no clique of linear size.
    It is a basic result from off-diagonal (hypergraph) Ramsey-theory that for every $s>q\ge 2$ we can color the edges of $K_n^{(q)}$ with red and blue such that there is no red $K_s^{(q)}$ and no blue $K_{f_s(n)}^{(q)}$, where $f_s(n)=o(n)$ (see for example the survey~\cite{mubayi2020survey}).
    If $H$ consists of the blue edges of such a coloring, then its complement contains no $K_{q+1}^{(q)}$, so $H$ is dense and $M_q^{(q)}(q+1)$-free, but its largest clique is $o(n)$, so $H$ does not have the fractional Helly property.
\end{proof}

\subsection{Homogenization}\label{subsec:superhom}

The main ingredient in the proof of Theorem~\ref{thm:main} is a Ramsey-type statement about the existence of highly homogeneous subhypergraphs in infinite hypergraphs; we state this as an independent lemma. For $1 \leq p \leq q$, and a sequence $(V_i)_{i \in \N}$ of sets, we call a $q$-tuple $(v_1, v_2, \ldots, v_q)$ an \emph{increasing $q$-tuple of $(V_i)_{i \in \N}$ starting with $(v_1, \ldots, v_p)$} if there exist indices $i_1< i_2 < \ldots < i_q$ such that $v_j \in V_{i_j}$. We say that a $q$-uniform hypergraph $H$ spanned by $\cup_i V_i$ is \emph{homogeneous with respect to an increasing $p$-tuple} $(v_1, \ldots, v_p)$, if either all increasing $q$-tuples starting with $(v_1, \ldots, v_p)$ are edges, or no increasing $q$-tuples starting with $(v_1, \ldots, v_p)$ are edges.  We say that $H[\cup_i V_i]$ is \emph{$p$-homogeneous}, if it is homogeneous with respect to every growing $p$-tuple. Finally, for an infinite sequence $(V_i)_{i \in \N}$ of sets, a \emph{subsequence of subsets} is an infinite sequence $(V_i')_{i \in \N}$ of sets such that there exists $i_1 < i_2 < \ldots$ with $V_j' \subset V_{i_j}$.

\begin{lemma}\label{lem:superhom}
    For every $q$ and sequence of integers $(n_i')_{i \in \N}$ there exists a sequence of integers $(n_i)_{i \in \N}$ such that if $(V_i)_{i \in \N}$ is a sequence of pairwise disjoint vertex sets of a $q$-uniform hypergraph $H$ with $|V_i| \geq n_i$ for all $i$, then we can find a subsequence of subsets $(V_i')_{i \in \N}$ of $(V_i)_{i \in \N}$ such that
    \begin{enumerate}
        \item We have $|V_i'| \geq n_i'$ for all $i$.
        \item $H[\cup_i V_i']$ is $(q-1)$-homogeneous.
    \end{enumerate}
\end{lemma}

The proof uses standard (hypergraph) Ramsey-type arguments.

\begin{proof}
Let $(n_i')_{i \in \N}$ be fixed. Without trying to optimize the values, let $n_i = n_i'$ for $i < q$ and $n_{i+1} = 2^{n_1 \cdots n_{i}}\max \{n'_1, \ldots, n'_{i+1}\}$ for $i \geq q$. Let $(V_i)_{i \in \N}$ be a sequence of finite pairwise disjoint subsets of $V(H)$ such that each $|V_i| \geq n_i$. Fix a well-ordering $T_1 \prec T_2 \prec \ldots$ of the $(v_{i_1}, \ldots, v_{i_{q-1}})$ increasing $(q-1)$-tuples of $(V_i)_{i \in \N}$ with the property that $(q-1)$-tuples ending at a lower level come first. More formally, if $T_i = (v_1, \ldots, v_{q-1})$ and $T_j = (w_1, \ldots, w_{q-1})$ are increasing $(q-1)$-tuples with $v_{q-1} \in V_s$ and $w_{q-1} \in V_t$, then $s < t$ implies $i < j$. We can homogenize $H[\cup_i V_i]$ with respect to all increasing $(q-1)$-tuples one by one by iteratively restricting to subsequences of subsets as follows. Let $(V^0_i)_{i \in \N}=(V_i)_{i \in \N}$.

\begin{claim}\label{cl:homStep}
    Let $(V^m_i)_{i \in \N}$ be a subsequence of subsets of $(V_i)_{i \in \N}$ and let $T = (v_1, \ldots, v_{q-1})$  be an increasing $(q-1)$-tuple with $v_{q-1} \in V^m_\ell$ and suppose that $H[\cup_i V^m_i]$ is already homogeneous with respect to all increasing $(q-1)$-tuples $S \in \cup_i V^m_i$with $S \prec T$. We can find an infinite subsequence of subsets $(V^{m+1}_i)_{i \in \N}$ of $(V^m_i)_{i \in \N}$ such that
    \begin{enumerate}
        \item $H[\cup_i V^{m+1}_i]$ is homogeneous with respect to $T$,
        \item If $V^{m+1}_i\subset V^m_j$, then $|V^{m+1}_i| \geq |V^m_j|/2$,
        \item If $i \leq \ell$, then $V^{m+1}_i = V^m_i$.
    \end{enumerate}
\end{claim}

\begin{proof}
For an $i > \ell$, call a set $V_i^m$ heavy if $|\{v \in V_i^m: T \cup \{v\}\in H\}| \geq \frac{1}{2}|V_i^m|$ and light otherwise. 

\textbf{Case 1: }There are infinitely many light sets. Let $i_{\ell + 1}, i_{\ell + 2}, \ldots$ be the indices of the light sets and let $V^{m+1}_j = \{v \in V^m_{i_j}: T \cup \{v\} \not \in H\}$ if $j > \ell$ and $V^{m+1}_j = V^m_j$ otherwise. 

\textbf{Case 2:} There are only finitely many light sets. Let $i_{\ell + 1}, i_{\ell + 2}, \ldots$ be the indices of the heavy sets and let $V^{m+1}_j = \{v \in V^m_{i_j}: T \cup \{v\} \in H\}$ if $j > \ell$ and $V^{m+1}_j = V^m_j$ otherwise. 
\end{proof}

We can iteratively apply this claim to every increasing $(q-1)$-tuple in the given order $\prec$ to find sequences $(V^0_i)_{i \in \N}, (V^1_i)_{i \in \N}, (V^2_i)_{i \in \N}, \ldots$ as follows. Suppose we already have $(V^m_i)_{i \in \N}$ and let $T_m$ be the first (according to  $\prec$) increasing $(q-1)$-tuple of $(V^m_i)_{i \in \N}$ such that $H[\cup_i V^m_i]$ is not homogeneous with respect to $T_m$. We can apply Claim~\ref{cl:homStep} to $(V^m_i)_{i \in \N}$ and $T_m$ to find a subsequence of subsets $(V^{m+1}_i)_{i \in \N}$ which is homogeneous with respect to $T_m$.

Repeating this process for every $m \in \mathbb{N}$, yields an infinite sequence of sequences $(V^0_i)_{i \in \N}, (V^1_i)_{i \in \N}, (V^2_i)_{i \in \N}, \ldots$ with the following properties. Most importantly, for every $i \in \mathbb{N}$ there is an $M(i)$ such that the sequence $(V^m_i)_{m \geq M(i)}$ is constant. This allows us to define $V'_i = V^{M(i)}_i$ for every $i \in \mathbb{N}$. The second property is that $H[\cup_i V'_i]$ is $(q-1)$-homogeneous.

For the third property, that $|V'_i| \geq n'_i$, define $N(n_1, \ldots n_s)$ as the number of increasing $(q-1)$-tuples of the finite sequence $(V_1, \ldots V_s)$ with $|V_i| = n_i$. To define $n_{s+1}$, observe that $n_{s+1} = 2^{N(n_1, \ldots, n_s)}\max \{n'_1, \ldots, n'_{s+1}\}\le 2^{n_1 \cdots n_s}\max \{n'_1, \ldots, n'_{s+1}\}$ is a good choice for $n_{s+1}$ for all $s \geq q$. Now if $V_i' \subset V_j$, then $|V_i'| \geq \frac{|V_j|}{2^{N(n_1, \ldots, n_j)}}$ by Claim~\ref{cl:homStep}. Since $j \geq i$, we have $|V_j| \geq 2^{N(n_1, \ldots, n_j)}n'_i$ and so $|V'_i| \geq n'_i$, finishing the induction step and thus the proof of Lemma~\ref{lem:superhom}.
\end{proof}

Lemma~\ref{lem:superhom} implies the existence of $p$-homogeneous subhypergraphs for every $1 \leq p < q$ by an easy induction argument.

\begin{corollary}\label{cor:superhom}
    For every $1 \leq p < q$ and a sequence of integers $(n_i')_{i \in \N}$ there exists a sequence of integers $(n_i)_{i \in \N}$ such that if $(V_i)_{i \in \N}$ is a sequence of pairwise disjoint vertex sets of a $q$-uniform hypergraph $H$ with $|V_i| \geq n_i$ for all $i$, then we can find a subsequence of subsets $(V_i')_{i \in \N}$ of $(V_i)_{i \in \N}$ such that
    \begin{enumerate}
        \item We have $|V_i'| \geq n_i'$ for all $i$.
        \item $H[\cup_i V_i']$ is $p$-homogeneous.
    \end{enumerate}
\end{corollary}

\begin{proof}
    Let $q$ be fixed. We proceed by backward induction on $p$. The $p=q-1$ case is Lemma~\ref{lem:superhom}. For the induction step, let $|V_i| = n_i$ be large enough. By the induction hypotheses, we can find a sequence of subsets $(V^*_i)_{i \in \N}$ of $(V_i)_{i \in \N}$ such that the $|V^*_i|$s are still large enough and $H[\cup_i V^*_i]$ is $p$-homogeneous. Let $H^\downarrow$ be the $p$-uniform hypergraph whose edge set is $\{T \subset \cup_i V^*_i: T $ is a increasing $p$-tuple, $ S \in H $ for all increasing $q$-tuple $S$ starting with $T\}$. Note that if a subsequence of subsets $(V_i')_{i \in \N}$ of $(V^*_i)_{i \in \N}$ is $(p-1)$-homogeneous in $H^\downarrow$, then it is also $(p-1)$-homogeneous in $H$. Now we can use Lemma~\ref{lem:superhom} with $q=p$ and the $p$-uniform hypergraph $H^\downarrow$ to find the next subsequence of subsets.
\end{proof}

\subsection{Putting it all together}\label{subsec:final}

\begin{proof}[Proof of Theorem~\ref{thm:main}]
    Let $H$ be an infinite $q$-uniform hypergraph satisfying the fractional Helly property and with arbitrarily large independent sets. We know that $H$ is $M_q^{(q)}(t)$-free with some finite $t \geq q$ by Claim~\ref{cl:Mqt}. 
    By applying Corollary~\ref{cor:superhom} with $p = 1$ and $n_i' = t$ for all $i$, we obtain a sequence $n_i$ such that from any pairwise disjoint independent sets $V_1, V_2, \ldots$ of $H$ with $|V_i| \geq n_i$, we can find disjoint independent sets $V_1', V_2', \ldots$ with $|V_i'| = t$ such that $H[\cup_i V_i']$ is $1$-homogeneous. (Such large enough disjoint $V_i$ necessarily exist because $H$ contains arbitrarily large independent sets.) 
    Since $H$ is $M_q^{(q)}(t)$-free, and each $V_i'$ is independent, every $V_i'$ has to contain a vertex $v_i$ such that no increasing $q$-tuple starting at $v_i$ is an edge, otherwise because of $p$-homogeneity we could obtain a hypergraph from $M_q^{(q)}(t)$. Therefore, the $\{v_i|i \in \N\}$ form an infinite independent set, finishing the proof.
\end{proof}

We remark that we do not even need $M_q^{(q)}(t)$-freeness in the proof, it is enough to assume $M_s^{(q)}(t)$-freeness for an arbitrary $s \geq q$; even $s = \aleph_0$ is enough, just at the end instead of every $V_i'$ containing a $v_i$, we get that all but a finite number of the $V_i'$ contain a suitable $v_i$. Thus, we have the following strengthening of Theorem~\ref{thm:main}.

\begin{theorem}\label{thm:mainGeneralized}
    Let $q$ and $t$ be positive integers and $H$ be a $q$-uniform, $M_{\aleph_0}^{(q)}(t)$-free hypergraph. If $H$ contains arbitrarily large finite independent sets, then $H$ contains an infinitely large independent set.
\end{theorem}

Rephrasing this in the contrapositive form in a more geometric setting: If a collection of objects $H$ satisfies some very weak form of the colorful Helly theorem, and among any infinitely many objects from $H$ some $q$ intersect, then there is a finite $p$ such that among any $p$ objects from $H$ some $q$ intersect. 
The very weak form that is already sufficient is that given $\aleph_0$ groups of $t$ objects such that any $q$ from different groups intersect, some $q$ from one group also intersect. We give an application of Theorem~\ref{thm:mainGeneralized} in the next section.

\section{Axis-parallel boxes and flats}\label{sec:boxes}

Theorem~\ref{thm:mainGeneralized} allows us to prove Keller-Perles-type infinite versions of the Alon-Kleitman theorem even if we do not have a suitable fractional Helly theorem for the given class of objects. For example, the following result of Chakraborty, Ghosh and Nandi is a corollary of known results and our Theorem~\ref{thm:mainGeneralized}.
They proved that if $\mathcal{F}$ is a family of axis-parallel boxes in $\mathbb{R}^d$ such that among every $\aleph_0$ of them some $2$ can be intersected with an axis-parallel $k$-flat, then all the members of $\mathcal{F}$ can be hit by finitely many axis-parallel $k$-flats.

Now we rephrase this result in our language.
Let $\mathcal{A}_{d,k}$ be the hypergraph whose vertices are axis-parallel boxes in $\mathbb{R}^d$ and whose edges represent families of boxes which can be intersected with a single axis-parallel $k$-flat. 

\begin{theorem}[Chakraborty, Ghosh and Nandi \cite{chakraborty2023stabbing}]\label{thm:boxesFlats}
    If $S \subset V(\mathcal{A}_{d,k})$ is such that $\mathcal{A}^{(2)}_{d,k}[S]$ has no infinitely large independent set, then $S$ can be covered with a finite number of edges of $\mathcal{A}_{d,k}$.
\end{theorem}

Note that the graph $\mathcal{A}^{(2)}_{d,k}$ does not satisfy the fractional Helly property if $k < d-1$, hence we cannot use Theorem~\ref{thm:main}.
However, we can prove Theorem~\ref{thm:boxesFlats} by applying Theorem~\ref{thm:mainGeneralized}, using the two below simple facts about finite families of boxes, whose proofs follow from standard methods, which can be found after the proof of Theorem~\ref{thm:boxesFlats}.

\begin{claim}\label{cl:MfreeBoxes}
    The graph $\mathcal{A}_{d,k}^{(2)}$ is $M_{s}^{(2)}(t)$-free for some large enough $t$ and $s$.
\end{claim}

\begin{claim}\label{cl:pqforBoxes}
    For every finite $p \geq 2$ there exists a $C<\infty$ with the property that if $S \subset V(\mathcal{A}_{d,k})$ is such that $\mathcal{A}_{d,k}^{(2)}[S]$ does not contain independent sets of size $p$, then $S$ can be covered with $C$ edges of $\mathcal{A}_{d,k}$.
\end{claim}

\begin{proof}[Proof of Theorem~\ref{thm:boxesFlats}]
    Assume that 
    $\mathcal{A}^{(2)}_{d,k}[S]$ has no infinitely large independent set. As $\mathcal{A}_{d,k}^{(2)}$ is $M_{\aleph_0}^{(2)}(t)$-free by Claim~\ref{cl:MfreeBoxes}, we can apply Theorem~\ref{thm:mainGeneralized} to conclude that there exists a finite $p$ such that $\mathcal{A}^{(2)}_{d,k}[S]$ has no independent sets of size $p$. But then $S$ can be covered with $C = C(d,k,p) < \infty$ edges by Claim~\ref{cl:pqforBoxes}.
\end{proof}

To prove the $M_{\aleph_0}^{(2)}(t)$-freeness of $\mathcal{A}^{(2)}_{d,k}$, we can use the following result of B{\'a}r{\'a}ny, Fodor, Mart{\'\i}nez-P{\'e}rez, Montejano,  Oliveros,  and P{\'o}r.

\begin{theorem}[B{\'a}r{\'a}ny et al.~\cite{barany2015fractional}]\label{thm:weakFHforBoxes}
    Let $S \subset V(\mathcal{A}_{d,0})$ and let $\alpha \in (1 - \frac{1}{d^2}, 1]$ be a real number. If $e(\mathcal{A}^{(2)}_{d,0}[S]) \geq \alpha\binom{|S|}{2}$, then there exists an edge of $\mathcal{A}_{d,0}[S]$ of size at least $(1 - d\sqrt{1-\alpha})|S|$.
\end{theorem}

\begin{proof}[Proof of Claim~\ref{cl:MfreeBoxes}]
    First, let $k=0$ and let $\mathcal{F}_1, \ldots, \mathcal{F}_s$ families of boxes each of size $t$ such that every two boxes from different families intersect. Then $\mathcal{F}_1\cup \ldots\cup \mathcal{F}_s$ is a family of $st$ boxes such that an $\alpha = \frac{\binom{s}{2}t^2}{\binom{st}{2}} = \frac{s-1}{s-\frac{1}{t}}$-fraction of the pairs intersect. As $\frac{s-1}{s-\frac{1}{t}} > \frac{s-1}{s} = 1 - \frac{1}{s} > 1 - \frac{1}{d^2}$ if $s$ is large enough, we can apply Theorem~\ref{thm:weakFHforBoxes} to conclude that there exists a large clique in the intersection graph of boxes. More precisely, we get a clique of size at least $(1-d\sqrt{1-\frac{s-1}{s-\frac{1}{t}}})st >(1-d\sqrt{1-(1 -\frac{1}{s})})st = (1-\frac{d}{\sqrt s})st > s$ if $s$ is large enough and $t\ge 2$. This clique must contain two intersecting boxes from the same family.

    We can reduce the $k > 0$ case to the $k = 0$ case with a Ramsey-type argument. Let $\mathcal{F}_1, \ldots, \mathcal{F}_s$ be families of boxes, each of size $t$, such that every two boxes from different families can be hit by an axis-parallel $k$-flat. 
    With a combination of bipartite and classical Ramsey theorems, for some $s'$ and $t'$ we can find $s'$ families of boxes, each of size at least $t'$, and an axis-parallel $k$-flat $W$ such that each pair of boxes from different families can be hit with a translate of $W$. We can have $s'$ and $t'$ still large enough if we had chosen $s$ and $t$ large enough. Project orthogonally to the orthogonal complement of $W$ and apply the $k=0$ case to find two boxes from the same family whose projections intersect. The original two boxes can be hit with a translate of $W$.
\end{proof}

A suitable analog of the Alon-Kleitman theorem is  true for $\mathcal{A}_{d,k}$. For that we first need the following result of Ding, Seymour and Winkler. For a hypergrpah $\mathcal{H}$ let $\lambda(\mathcal{H})$ be the maximum number $s$ such that there exist edges $E_1, \ldots, E_s$ of $\mathcal{H}$ with $(E_i \cap E_j ) \setminus \cup_{\ell \neq i,j}E_\ell\neq \emptyset$ for all $i, j \in \{1, \ldots, s\}$.

\begin{theorem}[Ding, Seymour and Winkler \cite{ding94}]\label{thm:pairedVC}
    There exists a function $f$ such that for any hypergraph $\mathcal{H}$ we have $\tau(\mathcal{H}) \leq f(\nu(\mathcal{H}), \lambda(\mathcal{H}))$.
\end{theorem}

\begin{proof}[Proof of Claim~\ref{cl:pqforBoxes}]
    Let us define the dual hypergraph $\mathcal{H}$ of $\mathcal{A}_{d,k}[S]$ where the vertices are axis-parallel $k$-flats and edges correspond to families of $k$-flats which intersect a given member of $S$. As $\mathcal{A}_{d,k}^{(2)}[S]$ does not contain independent sets of size $p$, we have $\nu(\mathcal{H}) \leq p$. Proving $\tau(\mathcal{H}) \leq C$ would mean $S$ can be covered with $C$ edges of $\mathcal{A}_{d,k}$. Due to Theorem~\ref{thm:pairedVC} it remains to prove the boundedness of $\lambda(\mathcal{H})$. For a contradiction, assume that for arbitrarily large $s$ there exists a family $\mathcal{F} = \{B_1, \ldots, B_s\}$ of axis-parallel boxes such that for every $i,j \in \{1, \ldots, s\}$ there exists an axis-parallel $k$-flat intersecting only $B_i$ and $B_j$ and no other member of $\mathcal{F}$.

    If $k=0$, then $\mathcal{F}$ is a family of $s$ boxes such that any two of them have a private intersection point not contained in any other box. In this case, $s$ cannot be too large, because otherwise we would have three boxes such that one contains the intersection of the two others as follows.
    
    As every box $B_i$ is the intersection of $2d$ halfspaces, we can write it as $B_i=\cap_{j=1}^d (H_i^{j+} \cap H_i^{j-})$ where $H_i^{j+}$ is a halfspace that contains all the points whose $j$-th coordinate is larger than some $x_i^{j+}$ and, similarly, $H_i^{j-}$ is a halfspace that contains all the points whose $j$-th coordinate is smaller than some $x_i^{j-}$.
    For any two boxes $i$ and $i'$, and any $j$, we can compare $x_i^{j+}$ with $x_{i'}^{j+}$, or $x_i^{j-}$ with $x_{i'}^{j-}$.
    This gives $2d$ different orderings on the boxes.
    
    We say that a triple of boxes are ordered consistently if their orders are the same for all the $2d$ orders. Via a repeated application of the Erd{\H{o}}s-Szekeres lemma on monotone subsequences \cite{erdos1935combinatorial}, if $s$ is large enough, there exists $3$ boxes which are ordered consistently. But then the intersection of the first box and the last box is contained in the middle box, which shows that they cannot have a private intersection point.
    
    Now we reduce the $k > 0$ case to the $k=0$ case. By applying Ramsey's theorem with $\binom dk$ colors, if $s$ is large enough, there exists an axis-parallel $k$-flat $W$ and subfamily $\mathcal{F}' \subseteq \mathcal{F}$ of (still large enough) size $s'$ such that for any two boxes $B_i, B_j \in \mathcal{F}'$ there exists a translate of $W$ intersecting only $B_i$ and $B_j$ and no other member of $\mathcal{F}'$. Project every member of $\mathcal{F}'$ to the orthogonal complement of $W$. As $s'$ is still large enough, we can use the $k=0$ case to find two boxes of $\mathcal{F}'$ whose projection does not have a private point. This leads to a contradiction, because then no translate of $W$ can hit only the two original boxes from $\mathcal{F}'$.
\end{proof}

\section*{Acknowledgement}

We would like to thank Andreas Holmsen for useful discussions.
We would also like to thank Chaya Keller for warning us that in the first version of this paper, we have stated the infinite heterochromatic results incorrectly, and for providing counterexamples for the incorrectly phrased statements.

\bibliographystyle{plain}
\bibliography{biblio}

\end{document}